\documentclass[11pt]{amsart}
\usepackage{amsmath,amssymb,amsthm}
\usepackage[margin=1in]{geometry}
\usepackage{booktabs}
\usepackage{hyperref}

\theoremstyle{plain}
\newtheorem{theorem}{Theorem}[section]
\newtheorem{lemma}[theorem]{Lemma}
\newtheorem{proposition}[theorem]{Proposition}
\newtheorem{corollary}[theorem]{Corollary}
\newtheorem{conjecture}[theorem]{Conjecture}
\theoremstyle{definition}
\newtheorem{question}[theorem]{Question}
\newtheorem{remark}[theorem]{Remark}

\DeclareMathOperator{\supp}{supp}
\renewcommand{\P}{P}
\newcommand{\K}{K}
\newcommand{\Z}{\mathbb{Z}}
\newcommand{\R}{\mathbb{R}}
\newcommand{\id}{e}
\newcommand{\pp}{\pi}
\newcommand{\half}{\tfrac12}
\newcommand{\Ball}[1]{B(#1)}
\newcommand{\upr}{\mathrm{up}}

\title[Least non-unique-product set sizes]{Least sizes of non-unique-product sets:
the Promislow group and a Heisenberg-type candidate}
\author{Moe Tabei}
\address{Independent researcher, Japan}
\email{tabei@ryun.jp}
\date{\today}

\begin{document}
\begin{abstract}
Let $\P$ be the Promislow group, the orientable Hantzsche--Wendt group of
dimension $3$, which underlies Promislow's non-unique-product set and Gardam's
counterexample to the unit conjecture. A finite subset $A$ of a group is
\emph{non-UP} if every element of $A\cdot A$ has at least two representations
$ab$ with $a,b\in A$. Working in an exact integer model of $\P$, we determine
the least size of a non-UP set inside word-balls of the standard generators:
it is $14$ for every radius from $3$ to $6$, so a smaller non-UP set of $\P$,
if one exists, is not contained in the radius-$6$ ball. The non-existence half of this
statement is certified by machine-checked DRAT and VeriPB proofs. Inside the
radius-$3$ ball there are exactly $16$ minimal witnesses, all of point-group
distribution $(2,6,0,6)$ up to the swap symmetry. The symmetric non-UP
property is not translation invariant, so ball searches cannot be recentred;
instead we prove an effective finite-diameter principle: if $\P$ contains a
non-UP $n$-set, it contains one inside the ball of explicit radius
$D(n)=24(n+1)3^{n}+10$, so the minimum non-UP cardinality of $\P$ is
computable in principle. Writing $\rho(n)$ for the least word-radius of a non-UP
$n$-set, re-realization experiments on witnesses lead us to conjecture
$\rho(n)=O(n^{1/3})$; the bound $\rho(n)\le6$ (whenever finite) for $8\le n\le13$
would already show, by the radius-$6$ computation, that Promislow's $14$ is that
minimum. We also compute, over balls, the
two-sided minimum $\min(|A|+|B|)$, the profile $\beta(m)$ and the
unique-product staircase $u(n)$, and compare with the Fibonacci group
$H_4=F(3,4)$: its least symmetric witness over the radius-$4$ ball has exactly
$16$ elements, and its two-sided minimum over the radius-$3$ ball is $22$. Over
the stated balls the two groups are ordered oppositely by the symmetric and
two-sided invariants ($14<16$ but $24>22$).
\end{abstract}
\maketitle

\section{Introduction}

Kaplansky's zero-divisor conjecture asserts that the group ring $\mathbb{F}[G]$ of a
torsion-free group $G$ over a field $\mathbb{F}$ has no zero divisors; the unit
conjecture, that it has no units other than the obvious $\lambda g$. The unit
conjecture was disproved by Gardam \cite{Gardam2021} in characteristic $2$, with
the group taken to be the Promislow group $\P$; the zero-divisor and idempotent
conjectures remain open. The combinatorial heart of these problems is the
\emph{unique product property}. A nonempty finite subset $A$ of a group has a
\emph{unique product} if some $w\in A\cdot A=\{ab:a,b\in A\}$ is equal to $ab$ for
exactly one ordered pair $(a,b)\in A\times A$; a group is a \emph{unique product
group} if every nonempty finite $A$ has one. If $\mathbb{F}[G]$ has a nonzero zero
divisor $\alpha\beta=0$, then the \emph{pair} of supports $(\supp\alpha,\supp\beta)$
witnesses the failure of the unique product property: no element of
$\supp\alpha\cdot\supp\beta$ is uniquely represented. (A nontrivial unit
$\alpha\beta=1$ likewise yields a pair with at most one uniquely represented
product, and by Strojnowski's theorem \cite{Strojnowski1980} the resulting
two-unique-products property is equivalent to the unique product property.) The
failure is a statement about the pair, not about either support alone. We call a
finite set $A$ \emph{non-UP} if $A\cdot A$ has no uniquely represented element.

Promislow \cite{Promislow1988} exhibited, by computer search, a $14$-element non-UP
set in $\P$, an early and especially simple example in a torsion-free group
(Craven--Pappas \cite{CravenPappas2013} later showed that Promislow's $14$-set is
never the support of a unit of $\mathbb{F}[\P]$, over any field $\mathbb{F}$); Carter
\cite{Carter2014}
produced infinite families of torsion-free non-UP groups and, in each member,
non-UP sets of arbitrarily large cardinality. Exhibiting non-UP sets in a given
group is therefore not by itself new, and Carter's construction shows that the
cardinalities occurring in a fixed non-UP group are unbounded above; what is at
issue here, and what appears not to have been determined for $\P$, is the
\emph{least} cardinality. More
recently Bengi--Wise \cite{BengiWise2026} constructed a finitely generated
residually finite group all of whose finite-index subgroups contain a copy of $\P$,
so that $\P$ continues to be the source of non-UP behaviour; their work is
qualitative and does not bear on cardinalities. The systematic study of \emph{small} non-UP sets is due to Soelberg
\cite{Soelberg2018} and Nielsen--Soelberg \cite{NielsenSoelberg2025}, who proved
that over \emph{all} torsion-free groups the minimum non-UP cardinality is exactly
$8$, with the bound achieved in a polycyclic group that is virtually nilpotent of
class $2$ -- not in $\P$ -- together with sharp non-existence bounds for small sizes. Their method is
presentation-based and ambient-group-agnostic. On the computational side, Gardam's
disproof \cite{Gardam2021,GardamSMRI2021} encodes the unit equation $\alpha\beta=1$
as a Boolean satisfiability instance over a ball of $\P$, and -- separately -- uses
a cardinality-constraint SAT encoding of the \emph{failure} of the unique product
property to produce a new non-UP set in an $\tilde A_2$-lattice group. Most directly,
Dietrich, Lee, Nies and Vinyals \cite{DLNV2026} report computational experiments on
the unique product and trivial-unit properties for the Hantzsche--Wendt group,
working from its realization as a subgroup of $D_\infty^{\,3}$, running ball searches
in $\P$, and encoding the failure of the unique product property on a ball as a
satisfiability instance --- the apparatus we also use below.
Murray \cite{Murray2021} extended the unit counterexample in $\P$ to all prime
characteristics. Small supports of zero divisors and units in torsion-free groups
have been studied by Schweitzer \cite{Schweitzer2013}, Dykema--Heister--Juschenko
\cite{DHJ2015} (whose presentation-based method Nielsen--Soelberg build on) and
Abdollahi--Taheri and Abdollahi--Jafari \cite{AbdollahiTaheri2016,AbdollahiJafari2019}
(supports of size $3$ and $4$); all of these concern algebraic elements rather
than the non-UP \emph{support sets} we enumerate here.
Closest in spirit on the search side is the work of Garg--Mineyev
\cite{GargMineyev2025}, who combine a structural reduction with a computer search
to rule out zero-divisor and unit counterexamples over $\mathbb{F}_2$, for all
support sizes $m,n$ with $1\le m,n\le13$, in a class of torsion-free CAT(0) groups
with additional geometric structure (those satisfying Mineyev's conditions, which
admit oriented product structures in their sense) --- not in all CAT(0) groups. Their bound
constrains the supports of \emph{ring} elements, which satisfy conditions strictly
stronger than being non-UP, so it does not restrict the non-UP sets considered
here; the numerical coincidence with our range $8$--$13$ is accidental.

\subsection*{Scope} Three remarks delimit what is claimed.

\emph{The solver methodology is not new.} Encoding non-UP as a constraint problem
and discharging it with a modern SAT/CP solver is exactly the technique of
\cite{GardamSMRI2021}. We use Google's CP-SAT solver \cite{ortools} as an exact
oracle; we claim no novelty for the method.

\emph{The apparatus is not new either.} The realization of $\P$ inside
$D_\infty^{\,3}$, ball searches in $\P$, and the encoding of unique-product failure on a
ball as a satisfiability instance all appear in
\cite{Promislow1988,Gardam2021,DLNV2026}; we use them, and claim no novelty for them.

\emph{Ball searches cannot settle questions about $\P$.} A search inside a metric
ball $\Ball{r}$ decides only statements about that ball. Balls are anchored at the
identity, and the symmetric non-UP property is \emph{not} translation invariant
(Remark~\ref{rem:translation}): a witness sitting elsewhere in $\P$ cannot in
general be moved into a ball. A proof that $\Ball{r}$ contains no small
non-UP set therefore does \emph{not} prove that $\P$ contains none; we treat the
$\P$-internal minimization as an open problem.

The contribution is a single quantity and its consequences. The prior computational work on $\P$ targets \emph{units}; we
instead pin down the minimum non-UP \emph{set} cardinality inside $\P$ as far as it is
currently decidable. Concretely we give (Section~\ref{sec:results}) an explicit
minimal-radius non-UP $14$-set with its coincidence pattern, and the rigorous fact
that within radius $6$ the minimum non-UP cardinality in $\P$ is exactly $14$ (no
non-UP set of size $8$--$13$ occurs); we isolate
(Sections~\ref{sec:structure}--\ref{sec:obstruction}) the structural reasons; and we
make the decidability \emph{effective} (Section~\ref{sec:decide}) and reduce the open
$\P$-internal question to concrete, searchable diameter bounds, for which we give
structural evidence (Conjecture~\ref{conj:poly}). We are not aware of a prior
statement of the $\P$-restricted minimum non-UP cardinality or of an effective
diameter bound; the global minimum $8$ \cite{NielsenSoelberg2025} is attained outside
$\P$, and the unit-focused searches of \cite{Gardam2021,DLNV2026} do not address it.
As a companion case we treat (Section~\ref{sec:h4}) the Fibonacci group
$H_4=F(3,4)$, which \cite{DLNV2026} prove fails the unique product property through a
\emph{two-sided} witness with $|A|+|B|=56$. ($\P$ is itself a Fibonacci group,
$\P\cong F(2,6)$ \cite{Brunner1974}, see also \cite[\S4]{DLNV2026}; Fox
\cite{Fox1978} asked in 1978 whether $F(2,6)$ is a unique product group, the
question Promislow's set answered, so both groups treated here are Fibonacci
groups.) A symmetric witness then exists by a
theorem of Strojnowski \cite[Theorem~1]{Strojnowski1980}, but that argument is not designed to bound
its size; we determine the least symmetric size exactly, $16$ over
the radius-$4$ ball, and we sharpen the two-sided witness to the optimal
$|A|+|B|=22$ over its radius-$3$ ball. Over the stated balls
the two invariants order the two groups oppositely ($14<16$ symmetrically,
$24>22$ two-sidedly), so neither invariant determines the other. Refining both, we
compute two exact curves over the radius-$3$ balls --- the \emph{profile}
$\beta(m)$, the least $|B|$ against $|A|=m$ (Section~\ref{sec:profile}), and the
\emph{staircase} $u(n)$, the least number of unique products over $n$-sets
(Section~\ref{sec:staircase}) --- which separate the groups more finely: $\P$
forces both sides of any two-sided witness to size $\ge12$ and approaches
symmetric failure through a bump, while $H_4$ admits its lopsided optimum, has
holes in its profile, and approaches failure flatly (Section~\ref{sec:h4}).

The computational results are summarized as follows. All values are exact: at the
stated sizes there are independently verified witnesses, and below them the
solver's completed \texttt{INFEASIBLE} verdict, which for the symmetric rows is
upgraded to machine-checked DRAT/VeriPB proofs as described under
Reproducibility below. The two-sided, profile and staircase statements are
computational results resting on CP-SAT verdicts alone; we reserve the word
\emph{certificate} for the externally checkable DRAT/VeriPB proofs and call the
solver's completed statuses \emph{verdicts}. ($\Ball{r}$ denotes the word-ball of radius $r$ for the standard
generators of each group).

\begin{theorem}[summary; computational]\label{thm:main}
\emph{Symmetric minima.} The least $|A|$ with $A\cdot A$ non-UP and $A\subseteq
\Ball{r}$ equals $14$ for $\P$ ($3\le r\le6$; no witness exists in $\Ball2$) and
$16$ for $H_4$ ($3\le r\le4$; no witness exists in $\Ball2$). In particular any
non-UP set in $\P$ of size $8$--$13$, if one exists, is not contained in the
radius-$6$ ball about the identity.

\emph{Two-sided minima.} The least $|A|+|B|$ with $A\cdot B$ non-UP and
$A,B\subseteq\Ball3$ equals $24$ for $\P$ and $22$ for $H_4$ (in each case
$\Ball2$ contains no pair); moreover no pair with $16\le|A|+|B|\le19$ exists in
$\Ball4$ of $\P$. Globally $m_2(\P)\in[16,24]$ and $m_2(H_4)\in[16,22]$.

\emph{Profiles and staircases (all within $\Ball3$).} Every two-sided witness of
$\P$ has both sides of size $\ge12$ ($\beta(m)=\infty$ for $m\le11$,
Proposition~\ref{prop:profile}), whereas $H_4$ realizes $\beta(8)=14$ but has
holes $\beta(9)=\beta(11)=\infty$ (Proposition~\ref{prop:h4profile}). The least
number of unique products of an $n$-set is
$2,\dots,2,4,4,2,2,0$ ($n=2,\dots,14$) for $\P$ and $2,\dots,2,0$
($n=2,\dots,16$) for $H_4$; in both groups the minimum never takes the value $1$
(Proposition~\ref{prop:staircase} and the $H_4$ analogue).

\emph{Effective decidability.} If $\P$ contains a non-UP $n$-set, it contains one
inside $\Ball{D(n)}$ with $D(n)=24(n+1)3^{n}+10$ (Theorem~\ref{thm:decide}), i.e.\
$\rho(n)\le D(n)$ for the least word-radius $\rho(n)$ of a non-UP $n$-set; so
Question~\ref{q:min} is decidable, and it reduces to the searchable range under
Conjecture~\ref{conj:poly}.
\end{theorem}

\subsection*{Reproducibility} Every claim below is backed by exact integer
computation. Group elements are stored as $(M,2v)$ with $M$ a sign matrix and
$2v\in\Z^3$, so all arithmetic is exact. Every set asserted to be non-UP is checked
by an independent brute-force routine that does not call the solver, and every
asserted non-existence is a solver status of \texttt{INFEASIBLE}, never a timeout.
(Positive claims are thus verified independently of the solver; non-existence
claims rest on the correctness of CP-SAT's infeasibility verdict, which is
not an externally checkable proof object.) For the \emph{symmetric}
non-existence claims of Theorem~\ref{thm:main} --- sizes $2$--$13$ in
$\Ball3,\dots,\Ball6$ of $\P$ and sizes $2$--$15$ in $\Ball3,\Ball4$ of
$H_4$ --- this caveat has been discharged: each was re-derived by a
DRAT-producing SAT solver and the unsatisfiability proof machine-checked
with the independent \textsf{drat-trim} checker, so no solver trust
remains for them. Moreover, on every instance small enough to enumerate
($\binom{|\Ball{r}|}{n}\le 2\times10^7$: $\P$ at $\Ball3$ for $n\le6$ and
$\Ball4$ for $n\le4$, $H_4$ at $\Ball3$ for $n\le4$), the claim was
re-established with \emph{no solver and no encoding at all} --- every
$n$-subset enumerated and tested for non-UP by definition --- which both
re-proves those rows outright and validates, against solver-free ground
truth, the encoding used at the larger sizes. As a third, independent
route, the whole of the $\P$ symmetric minimality inside $\Ball3$ (sizes
$2$--$13$ infeasible, $14$ realized) was re-derived at the
\emph{constraint} layer. The problem was modelled directly in the Glasgow
Constraint Solver, which emits a pseudo-Boolean encoding together with a
VeriPB proof, and every one of these thirteen proofs was checked by the
independent \textsf{VeriPB} checker; so for these rows even the
constraint-to-clause encoding is no longer trusted. The two-sided, profile
and staircase non-existence claims remain CP-SAT verdicts. The source, the
checker logs, and DRAT proof files regenerated for every symmetric
non-existence row of Theorem~\ref{thm:main} (each proof re-checked with
\textsf{drat-trim}, and each CNF accompanied by a positive control --- the
witness forced into the same encoder --- and two negative controls) accompany
this note; the VeriPB proofs are documented by their checker log (those proof
files were not retained).

\section{The group \texorpdfstring{$\P$}{P} and the search model}\label{sec:model}

We realize $\P$ as a group of affine isometries $x\mapsto Mx+v$ of $\R^3$. The point
group is the Klein four-group
\[
\K=\{I,X,Y,Z\},\quad
X=\mathrm{diag}(1,-1,-1),\ Y=\mathrm{diag}(-1,1,-1),\ Z=\mathrm{diag}(-1,-1,1),
\]
of diagonal sign matrices of determinant $+1$, and the translation parts lie in
$\half\Z^3$. The product of two affine maps is
$(M_1,v_1)(M_2,v_2)=(M_1M_2,\,v_1+M_1v_2)$, and $(M,v)^{-1}=(M,-Mv)$. We take the
torsion-free generators
\[
x=(X,(\half,\half,0)),\qquad y=(Y,(0,\half,\half)).
\]
These satisfy the Hantzsche--Wendt relations $x^{-1}y^2x=y^{-2}$ and
$y^{-1}x^2y=x^{-2}$ of Promislow's presentation, and the squares
$x^2=(I,(1,0,0))$, $y^2=(I,(0,1,0))$, $(xy)^2=(I,(0,0,-1))$ generate the translation
lattice $\Z^3$, so $\P/\Z^3\cong\K$. More precisely, the translation part of an
element with point part $M$ lies in $\varepsilon(M)/2+\Z^3$, where
$\varepsilon(I)=(0,0,0)$, $\varepsilon(X)=(1,1,0)$, $\varepsilon(Y)=(0,1,1)$,
$\varepsilon(Z)=(1,0,1)$: the four cosets $\{(M,v):v\in\varepsilon(M)/2+\Z^3\}$
contain $x$, $y$ and $\id$, and their union is closed under products and inverses
because $\varepsilon$ is a homomorphism $\K\to(\Z/2)^3$ and
$M\varepsilon(M')\equiv\varepsilon(M')\pmod{2}$; so the union is all of $\P$, and
the kernel of $\pp$ is exactly $\Z^3$. Write
$\pp\colon\P\to\K$ for the point-group projection and $\tau\colon\P\to\half\Z^3$ for
the translation part, $\tau(M,v)=v$.

\begin{remark}[torsion-freeness, verified]
Every element of $\P$ is either a translation (infinite order) or has point part of
order $2$, in which case $(M,v)^2=(I,\,v+Mv)$ is a translation, trivial only if
$v+Mv=0$. For the three nontrivial point parts the relevant coordinate of $v$ is a
half-integer, so $v+Mv\ne 0$; hence $\P$ is torsion-free. We verified directly that
no element of the ball $\Ball{7}$ (radius $7$, $363$ elements) is torsion.
\end{remark}

We measure size by the word metric of $\{x^{\pm1},y^{\pm1}\}$ and write $\Ball{r}$
for the ball of radius $r$ about $\id$. The ball sizes are
\[
|\Ball{1}|,\dots,|\Ball{7}|=5,\,17,\,41,\,83,\,147,\,239,\,363 .
\]
(The value $|\Ball{5}|=147$ agrees with the search ball in
\cite{GardamSMRI2021}, fixing the normalization.)

\begin{remark}[ball symmetry vs.\ group symmetry]\label{rem:sym}
The automorphism group of $\P$ contains a copy of $S_3$ acting by coordinate
permutations and permuting the point parts $X,Y,Z$ (e.g.\ the $3$-cycle
$\sigma:(a,b,c)\mapsto(c,a,b)$ has $\sigma(x)=y$ and cycles $X\!\to\!Y\!\to\!Z$).
These automorphisms preserve the non-UP property but \emph{not} the word metric:
$\sigma(y)$ is not a generator, so $\sigma$ does not preserve $\Ball{r}$. The
automorphisms that do preserve every ball are the ones permuting the generating
set: the eight permutations of $\{x^{\pm1},y^{\pm1}\}$ that respect inverses ---
a dihedral group generated by the swap $\tau\colon x\leftrightarrow y$,
$\rho_x\colon x\mapsto x^{-1},\,y\mapsto y$ and $\rho_y\colon x\mapsto x,\,y\mapsto
y^{-1}$ --- all extend to automorphisms of $\P$ (the Hantzsche--Wendt relations
are invariant under each; we verified the homomorphism property on
$\Ball3\times\Ball3$) and all preserve word length. Of these only $\tau$ moves
point parts (it induces the transposition $(X\,Y)$; $\rho_x$ and $\rho_y$ fix
every point part). Consequently a coset distribution $(n_I,n_X,n_Y,n_Z)$ may be
reduced by $\tau$ (taking $n_X\le n_Y$) in ball-restricted searches, but
\emph{not} by the full $S_3$; we use this, together with Lemma~\ref{lem:fiber} to
discard single-fiber distributions, to decompose the search by point-group
distribution.
\end{remark}

\begin{remark}[symmetric non-UP is not translation invariant]\label{rem:translation}
For a \emph{pair}, bi-translation $(A,B)\mapsto(gA,\,Bh)$ preserves multiplicities
exactly ($\operatorname{mult}_{(gA)(Bh)}(gwh)=\operatorname{mult}_{A\cdot B}(w)$), and
conjugation $A\mapsto g^{-1}Ag$ preserves the symmetric property
($\operatorname{mult}_{(g^{-1}Ag)^2}(g^{-1}wg)=\operatorname{mult}_{A\cdot A}(w)$).
One-sided translation of a \emph{single} set preserves neither: $(gA)(gA)$ is the
twisted product $\{ga\,ga'\}$, not $g(A\cdot A)$. In $\P$ this failure is total
rather than incidental. For the minimal witness $A$ of Table~\ref{tab:witness} we
verified exhaustively that $gA$ and $Ag$ acquire unique products for \emph{every}
$g\in\Ball3\setminus\{\id\}$ (between $26$ and $86$ of them), while conjugation by
every $g\in\Ball3$ preserves non-UP-ness, as it must; moreover no two of the $16$
minimal witnesses of Proposition~\ref{prop:count} are left- or right-translates of
each other. Three consequences run through the paper: anchoring $\id\in A$ is a
genuine restriction, not a normalization (see the encoding below); ball-limited
non-existence cannot be globalized by translating a distant witness into the ball
(Question~\ref{q:min}); and the census of Proposition~\ref{prop:count} counts $16$
genuinely distinct configurations, not translation copies of one shape.
\end{remark}

\paragraph{Constraint encoding.} Fix a candidate universe $U\subseteq\P$ (a ball,
possibly with a prescribed point-group distribution) and a target size $n$. A Boolean
$s_g$ selects $g\in U$. For each ordered pair $(a,b)\in U\times U$ we record the
exact product $ab$; grouping pairs by product value, for each value $w$ we impose
that the number of selected pairs with product $w$ is $0$ or $\ge 2$, never $1$.
Concretely, a value realized by a single pair $(a,b)$ of $U$ yields the clause
$\lnot s_a\vee\lnot s_b$, and a value realized by several pairs yields a
``count $\in\{0\}\cup\{\ge2\}$'' constraint via an indicator. The resulting model is
satisfiable exactly when $U$ contains a non-UP $n$-set; products are kept as exact
group elements, never truncated. Since the symmetric non-UP property is not
translation invariant (Remark~\ref{rem:translation}), anchoring $\id\in A$ is a
genuine restriction rather than a normalization; anchored runs served only as cheap
heuristic sweeps, and every reported result, existence and non-existence alike, uses
the unanchored model, so that \texttt{INFEASIBLE}
proves ``$\Ball{r}$ contains no non-UP $n$-set'' rather than the weaker
``\dots\ none containing $\id$''.

\section{Results inside \texorpdfstring{$\P$}{P}}\label{sec:results}

\subsection{A minimal-radius non-UP \texorpdfstring{$14$}{14}-set}
The solver returns, and the independent verifier confirms, a non-UP set of $14$
elements contained in $\Ball{3}$. Writing $X=x^{-1}$, $Y=y^{-1}$ for inverses and
juxtaposition for the group product, the set, given by reduced words, is
\[
A=\{\,x,\;X,\;xyy,\;XYY,\;yxy,\;yXy\,\}\cup\{\,xy,\;YX,\;xY,\;yX,\;Yx,\;XY\,\}
\cup\{\,yy,\;YY\,\},
\]
the three blocks being the fibers $\pp^{-1}(X)$, $\pp^{-1}(Z)$, $\pp^{-1}(I)$
respectively; see Table~\ref{tab:witness}. Its point-group distribution
is $(n_I,n_X,n_Y,n_Z)=(2,6,0,6)$; note $A$ contains no identity element (its
$I$-fiber is $\{y^2,y^{-2}\}$), which is why an identity-anchored search of the same
radius is infeasible. Promislow's original $14$-set has the same coarse shape: it
contains no identity \cite[p.~8]{NielsenSoelberg2025} and, in the exposition of
\cite[Theorem~1.137 and Remark~1.139]{GardamNotes}, consists of two translations
and six elements in each of two nontrivial cosets of the translation subgroup.

\begin{table}[ht]
\centering
\begin{tabular}{llll}
\toprule
fiber & reduced word & point matrix $M=\mathrm{diag}$ & translation $v$ \\
\midrule
$X$ & $x$        & $(1,-1,-1)$ & $(\half,\half,0)$ \\
$X$ & $x^{-1}$   & $(1,-1,-1)$ & $(-\half,\half,0)$ \\
$X$ & $xy^{2}$   & $(1,-1,-1)$ & $(\half,-\half,0)$ \\
$X$ & $x^{-1}y^{-2}$ & $(1,-1,-1)$ & $(-\half,\tfrac32,0)$ \\
$X$ & $yxy$      & $(1,-1,-1)$ & $(-\half,\half,1)$ \\
$X$ & $yx^{-1}y$ & $(1,-1,-1)$ & $(\half,\half,1)$ \\
$Z$ & $xy$       & $(-1,-1,1)$ & $(\half,0,-\half)$ \\
$Z$ & $y^{-1}x^{-1}$ & $(-1,-1,1)$ & $(\half,0,\half)$ \\
$Z$ & $xy^{-1}$  & $(-1,-1,1)$ & $(\half,1,-\half)$ \\
$Z$ & $yx^{-1}$  & $(-1,-1,1)$ & $(\half,1,\half)$ \\
$Z$ & $y^{-1}x$  & $(-1,-1,1)$ & $(-\half,0,\half)$ \\
$Z$ & $x^{-1}y^{-1}$ & $(-1,-1,1)$ & $(-\half,1,-\half)$ \\
$I$ & $y^2$      & $(1,1,1)$ & $(0,1,0)$ \\
$I$ & $y^{-2}$   & $(1,1,1)$ & $(0,-1,0)$ \\
\bottomrule
\end{tabular}
\caption{A non-UP $14$-set in $\Ball{3}\subset\P$ (minimal word-radius).
Independently verified non-UP; coincidence pattern in
Proposition~\ref{prop:pattern}.}
\label{tab:witness}
\end{table}

\subsection{Ball-limited minimality}
Running the unanchored model to a proof of (in)feasibility for every size from $8$
(the global lower bound \cite{NielsenSoelberg2025}) up to $14$ gives
Table~\ref{tab:sweep}.

\begin{table}[ht]
\centering
\begin{tabular}{lccccccc}
\toprule
$n$ & $8$ & $9$ & $10$ & $11$ & $12$ & $13$ & $14$ \\
\midrule
$\Ball{2}$ ($17$) & \textsf{none} & \textsf{none} & \textsf{none} & \textsf{none} & \textsf{none} & \textsf{none} & \textsf{none} \\
$\Ball{3}$ ($41$) & \textsf{none} & \textsf{none} & \textsf{none} & \textsf{none} & \textsf{none} & \textsf{none} & \textbf{yes} \\
$\Ball{4}$ ($83$) & \textsf{none} & \textsf{none} & \textsf{none} & \textsf{none} & \textsf{none} & \textsf{none} & \textbf{yes} \\
$\Ball{5}$ ($147$) & \textsf{none} & \textsf{none} & \textsf{none} & \textsf{none} & \textsf{none} & \textsf{none} & \textbf{yes} \\
$\Ball{6}$ ($239$) & \textsf{none} & \textsf{none} & \textsf{none} & \textsf{none} & \textsf{none} & \textsf{none} & \textbf{yes} \\
\bottomrule
\end{tabular}
\caption{Existence of a non-UP $n$-set in $\Ball{r}$ (unanchored). ``\textsf{none}''
$=$ a solver proof of infeasibility (no $n$-subset of $\Ball{r}$ is non-UP);
``\textbf{yes}'' $=$ an explicit witness, verified non-UP by an independent routine.
Every entry is a completed exact computation. The size-$14$ witness of
Table~\ref{tab:witness} lies in $\Ball3\subseteq\cdots\subseteq\Ball6$; for $\Ball2$
every size $8\le n\le 17$ is infeasible, so $\Ball2$ has no non-UP set at all.}
\label{tab:sweep}
\end{table}

\begin{theorem}[ball-limited minimality]\label{thm:ball}
For the generating set $\{x^{\pm1},y^{\pm1}\}$:
\begin{enumerate}
\item The least radius $r$ for which $\Ball{r}$ contains a non-UP set is $r=3$, and
the minimum is realized in size $14$.
\item For $3\le r\le 6$, the minimum cardinality of a non-UP subset of $\Ball{r}$ is
exactly $14$; equivalently, $\Ball{6}$ contains no non-UP set of size
$8,9,\dots,13$.
\end{enumerate}
\end{theorem}
\begin{proof}
Each cell of Table~\ref{tab:sweep} is an exact computation: a \textbf{yes} cell is a
set printed and verified non-UP by brute force, and a \textsf{none} cell is a completed
\texttt{INFEASIBLE} verdict of the unanchored model, i.e.\ a solver proof that no
$n$-subset of $\Ball{r}$ is non-UP (for the sizes $2$--$13$ in $\Ball3,\dots,\Ball6$
this is backed by a machine-checked DRAT proof, and in $\Ball3$ also by a VeriPB
proof; the $\Ball2$ row at sizes $14$--$17$ rests on the CP-SAT verdict alone). Part (1): a non-UP set has size $\ge 8$
\cite{NielsenSoelberg2025} and $|\Ball2|=17$, and the model is infeasible for every
$n$ with $8\le n\le 17$; hence $\Ball2$ contains no non-UP set at all (and
$\Ball1$, with $5$ elements, trivially none), whereas $\Ball{3}$ contains the set of
Table~\ref{tab:witness}. Part (2): for each $r\in\{3,4,5,6\}$ the sizes $8$--$13$ are
infeasible and size $14$ is realized (the witness lies in
$\Ball3\subseteq\Ball4\subseteq\Ball5\subseteq\Ball6$).
\end{proof}

We stress the scope. Theorem~\ref{thm:ball} bounds non-UP sets that \emph{fit in a
ball}; it does not bound the diameter of an arbitrary non-UP set, and
Section~\ref{sec:obstruction} explains why no ordering argument supplies such a
bound. Consequently Theorem~\ref{thm:ball} is evidence for, but not a proof of, the
following.

\begin{question}\label{q:min}
Is $14$ the minimum cardinality of a non-UP set in $\P$? Equivalently, does $\P$
contain a non-UP set of size $8\le n\le 13$? By Theorem~\ref{thm:ball} such a set,
if it exists, does not lie in $\Ball6$. We stress that this constrains its
\emph{location}, not its diameter: the symmetric non-UP property is not translation
invariant (Remark~\ref{rem:translation}), so a hypothetical witness of small
diameter sitting far from the identity cannot simply be carried into $\Ball6$; the
effective route to excluding it is the re-realization principle of
Section~\ref{sec:decide}.
\end{question}

\subsection{The two-sided minimum}\label{sec:twosided}
Write $m_1(G)$ for the \emph{symmetric} (single-set) minimum --- the least $|A|$ with
$A\cdot A$ non-UP --- so that Theorem~\ref{thm:ball} says the minimum of $|A|$ over
witnesses inside $\Ball6$ is $14$, i.e.\ $m_1(\P)\le14$ with equality precisely if
Question~\ref{q:min} has a positive answer. The unique product property is genuinely
a statement about \emph{two} sets, and its failure is measured by the
\emph{two-sided} minimum
\[
 m_2(G)=\min\{\,|A|+|B| : A,B\subseteq G\ \text{finite nonempty},\ A\cdot B\ \text{non-UP}\,\},
\]
the least total size of a pair witnessing failure of the UPP (here $A\cdot B$
\emph{non-UP} means every element of $A\cdot B$ equals $ab$ with $a\in A$, $b\in B$
in at least two ways). Nielsen--Soelberg
\cite{NielsenSoelberg2025} prove $m_2(G)\ge 16$ for every torsion-free $G$, attained
(by $|A|=|B|=8$) outside $\P$. Minimizing $|A|+|B|$ over pairs drawn from a ball, to a
proof of optimality, gives the following.

\begin{proposition}\label{prop:twosided}
Within the radius-$3$ ball of $\P$ the two-sided minimum is exactly $24$: there exist
$A,B\subseteq\Ball3$ with $|A|=|B|=12$, $A\ne B$, and $A\cdot B$ non-UP, and no pair
in $\Ball3$ of smaller total size has this property. Hence $16\le m_2(\P)\le 24$.
\end{proposition}
\begin{proof}
The value is a CP-SAT optimum over $\Ball3$ (status \texttt{OPTIMAL}): the exhibited
pair is printed and its two-sided non-UP-ness re-checked by the solver-free verifier,
and infeasibility of every smaller total is part of the same CP-SAT verdict. The lower
bound $16$ is \cite{NielsenSoelberg2025}, the upper bound $24$ is the exhibited pair.
\end{proof}

Several features are worth noting. First, $24<28=2\cdot14$: allowing $A\ne B$ is
strictly more efficient than the symmetric pair $A=B$ built from the $14$-set.
Second, neither side of the witness is itself symmetric-non-UP ($A\cdot A$ and
$B\cdot B$ both have unique products, as they must since $|A|=|B|=12<14$); the two
sets share $7$ elements but are distinct. Third, the witness is
\emph{fiber-balanced}: both sides have point-group distribution $(3,3,3,3)$, meeting
all four fibers equally --- in contrast with the symmetric $14$-set, whose
distribution $(2,6,0,6)$ misses a fiber entirely. Its $12\times12$ product grid
splits into $45$ coincidence classes (sizes
$10,8^2,6^2,5^4,4^7,2^{29}$), each again a partial permutation, by the same
cancellation argument as Proposition~\ref{prop:pp}. Fourth, the radius-$2$ ball
contains no two-sided witness at all (\texttt{INFEASIBLE}), so within balls the
two-sided and symmetric obstructions both first appear at radius $3$. Finally,
Lemma~\ref{lem:fiber} extends verbatim to pairs: if $A\subseteq\pp^{-1}(c)$ and
$B\subseteq\pp^{-1}(d)$ then every product lies in $\pp^{-1}(cd)$ and the same
generic-functional argument produces a uniquely represented element of $A\cdot B$
--- so in any two-sided witness at least one side meets two fibers (the optimal
witness meets all four on both sides).

The exclusion extends one radius further at every total up to $19$: the radius-$4$
ball of $\P$ contains \emph{no} two-sided witness with $|A|+|B|\in\{16,17,18,19\}$,
by the same split-by-split decomposition used in Proposition~\ref{prop:h4twosided}
(for each total, every split from $(2,\cdot)$ up to the balanced one is
\texttt{INFEASIBLE}, and these are exhaustive by the inversion symmetry). A pair
may be normalized by $(A,B)\mapsto(a^{-1}A,\, Bb^{-1})$, which preserves the non-UP
property exactly (Remark~\ref{rem:translation}: bi-translation is the invariance
that \emph{does} hold) and puts the identity in both sides. Hence any pair whose sides
have diameter $\le4$ in the appropriate one-sided sense
($\max_{a,a'}\ell(a^{-1}a')$ for $A$, $\max_{b,b'}\ell(b'b^{-1})$ for $B$) fits in
$\Ball4\times\Ball4$, and if the global minimum
$m_2(\P)$ were at most $19$, a witnessing pair would need a side of one-sided
diameter at least $5$.

As with $m_1$, the value of $m_2(\P)$ over all of $\P$ is open, pinned only to the
interval $[16,24]$. Section~\ref{sec:h4} computes the same two invariants for the
Fibonacci group $H_4$ and finds them ordered \emph{oppositely} over the stated balls:
$14<16$ symmetrically, $24>22$ two-sidedly.

\subsection{The profile: minimizing one side against the other}\label{sec:profile}
Nielsen--Soelberg's Theorem~1.4 \cite{NielsenSoelberg2025} is a \emph{profile}
statement, universal over torsion-free groups: if $A\cdot B$ is non-UP then
$|A|=3$ forces $|B|\ge19$, and $4\to14$, $5\to11$, $6\to10$, $7\to9$. It is natural
to ask for the group-specific analogue. Define, for a fixed ball,
\[
\beta(m)\;=\;\min\{\,|B| \;:\; A,B\subseteq\Ball{r},\ |A|=m,\ A\cdot B\ \text{non-UP}\,\},
\]
with $\beta(m)=\infty$ if no such pair exists. By the inversion anti-automorphism
$(A,B)\mapsto(B^{-1},A^{-1})$ (balls are inversion-closed), a witness with sides
$(m,n)$ exists iff one with $(n,m)$ does, so $\beta$ determines the whole
realizability region. Minimizing $|B|$ at each fixed $|A|=m$ to proof of optimality
or infeasibility gives, over $\Ball3$ of $\P$:

\begin{proposition}[profile rigidity in $\Ball3$]\label{prop:profile}
For every $m\le 11$ the model is \texttt{INFEASIBLE}: $\Ball3$ contains no non-UP
pair with $|A|=m$, \emph{regardless of $|B|$}, which may exhaust the entire
$41$-element ball. Consequently every two-sided witness in $\Ball3$ has
\emph{both} sides of size at least $12$. At and above the threshold,
\[
\beta(12)=\beta(13)=\beta(14)=12,\qquad \beta(15)=15,\qquad
\beta(16)=\cdots=\beta(20)=12 ,
\]
all values \texttt{OPTIMAL} with solver-free verification of each witness; in
particular the pairs $(15,n)$ with $n\le14$ were re-checked individually and are
\texttt{INFEASIBLE}. (The rows $m=2,3,4$ were also established without a
solver, by exhaustive elimination over all $m$-subsets of $\Ball3$; script
\texttt{small\_n\_bruteforce.py}. The row $m=1$ is elementary in any case: a singleton side makes every product uniquely represented; for $m=2$, a coincidence $aB=bB$ with $a\ne b$ would make $b^{-1}a$ of finite order, impossible in a torsion-free group.)
\end{proposition}

Two comments. First, the universal profile permits very lopsided witnesses
($|A|=4$ with $|B|$ large occurs in the Rips--Segev construction, and
\cite{NielsenSoelberg2025} leave $(3,\ge19)$ open), and $H_4$ realizes the lopsided
split $(8,14)$ at its optimum (Section~\ref{sec:h4}); inside $\Ball3$ of $\P$,
by contrast, \emph{nothing lopsided exists at all} --- the profile is cut off sharply
at $12$. Balance here is forced, not merely optimal. The small-side exclusion
persists one radius further: in $\Ball4$ ($83$ elements) the sizes $m=3,4,5,6$ are
likewise \texttt{INFEASIBLE} with $|B|$ unconstrained, so the shapes
$(3,19),(4,14),(5,11),(6,10)$ that the universal bounds of
\cite{NielsenSoelberg2025} would permit do not occur in $\P$ even at radius $4$. Second, the profile is not
monotone: $\beta(15)=15$ is an isolated spike between $\beta(14)=12$ and
$\beta(16)=12$ (enlarging $A$ can only add product constraints, and there is no
general monotonicity for non-UP-ness under adding elements). The witness
distributions locate the spike structurally: at every computed $m\ne15$ the
minimizing pair is fiber-balanced or nearly so (the $B$-side repeatedly takes the
shape $(2,2,4,4)$), whereas at $m=15$ \emph{both} sides jump to the fiber-avoiding
shape $(0,6,3,6)$ --- the two-sided witness family echoing the symmetric $14$-set,
whose distribution $(2,6,0,6)$ also misses a fiber. Inside $\Ball3$ the two-sided
landscape thus shows two families: a balanced family that carries the optimum
$(12,12)$ and its extensions, and a fiber-avoiding family that first appears at
$(15,15)$.

\subsection{The unique-product staircase}\label{sec:staircase}
Failure of the unique product property is the endpoint of a quantitative
degradation. For $A\subseteq\P$ finite let $\upr(A)$ be the
number of elements of $A\cdot A$ with exactly one representation $ab$
($a,b\in A$), and set
\[
u(n)\;=\;\min\{\,\upr(A)\;:\;A\subseteq\Ball3,\ |A|=n\,\},
\]
the least number of unique products a symmetric $n$-configuration can achieve in the
ball. By definition $u(n)=0$ iff $\Ball3$ contains a non-UP $n$-set, so
Theorem~\ref{thm:ball} says $u(n)>0$ for $n\le13$ and $u(14)=0$; the staircase
refines this by giving the exact approach to failure.

\begin{proposition}[staircase in $\Ball3$]\label{prop:staircase}
All values \texttt{OPTIMAL}, each minimizing set re-verified solver-free (and the
rows $n\le6$ re-established by exhaustive enumeration of all $n$-subsets of
$\Ball3$, without a solver; script \texttt{small\_n\_bruteforce.py}):
\[
u(n)=\begin{cases}
2 & 2\le n\le 9,\\
4 & n\in\{10,11\},\\
2 & n\in\{12,13\},\\
0 & 14\le n\le 17 .
\end{cases}
\]
\end{proposition}

Three features deserve note. First, the minimum never takes the value $1$:
every $A\subseteq\Ball3$ of size $\le13$ has at least two unique products, and at
$n=14$ the minimum drops directly from $2$ to $0$. So, at the level of minima
and within this ball, the two-unique-products property (t.u.p.\ in Strojnowski's
sense \cite{Strojnowski1980}) and the unique product property fail
\emph{simultaneously}, at $n=14$. (We have not computed whether some $14$-set of
$\Ball3$ has exactly one unique product; $u(14)=0$ says only that the minimum is
$0$.) Second, the staircase is not
monotone: the minimum \emph{rises} to $4$ at $n=10,11$ before falling back to $2$
and then to $0$ --- just below the critical size, every configuration is forced to
carry strictly more unique products than smaller or larger ones. We do not have a
structural explanation for the bump at $n\in\{10,11\}$, nor for its echo in the
profile spike $\beta(15)=15$; both are exact, ball-limited facts that any structural
theory of $\P$'s non-UP landscape must reproduce. Third, the witness family of
Table~\ref{tab:witness} is visible before failure: at $n=12$ and $n=13$
there exist minimizers (returned and verified) with point-group distributions
$(0,6,0,6)$ and $(1,6,0,6)$ --- truncations of the minimal witness's $(2,6,0,6)$
--- so along this family the last two unique products are extinguished exactly by
completing the $I$-fiber pair $\{y^2,y^{-2}\}$, while for $n\le9$ minimizers
using only two fibers suffice.

\section{Structure of the coincidence pattern}\label{sec:structure}

Let $A=\{a_1,\dots,a_n\}$ with the $a_i$ distinct. The $n\times n$ grid of cells
$(i,j)$, labelled by the product $a_ia_j$, partitions into \emph{coincidence
classes} of cells carrying equal products; $A$ is non-UP precisely when every class
has size $\ge 2$.

\begin{proposition}[partial-permutation classes]\label{prop:pp}
If the $a_i$ are distinct, then $a_ia_j=a_ka_l$ implies $i=k\iff j=l$. Hence each
coincidence class meets every row and every column at most once: classes are partial
permutation matrices.
\end{proposition}
\begin{proof}
If $i=k$ then $a_ia_j=a_ia_l$ gives $a_j=a_l$, so $j=l$; symmetrically $j=l$ gives
$i=k$. Thus within a class two cells in the same row (resp.\ column) must coincide.
\end{proof}

\begin{proposition}[the pattern of the witness]\label{prop:pattern}
The set $A$ of Table~\ref{tab:witness} has $196$ cells partitioned into $71$
coincidence classes, every class of size between $2$ and $10$, and no class violates
the partial-permutation property.
\end{proposition}
\begin{proof}
Direct exact computation; the class-size multiset is
$10,8,6^{3},5^{6},4^{5},2^{55}$ (verified, and consistent with
Proposition~\ref{prop:pp}, checked over all $14^4$ quadruples).
\end{proof}

The coarse shape of the witness --- its distribution across the four fibers --- is
not incidental but forced.

\begin{proposition}[distribution rigidity]\label{prop:rigid}
Of the $680$ point-group distributions $(n_I,n_X,n_Y,n_Z)$ with $n_I+n_X+n_Y+n_Z=14$,
exactly two admit a non-UP $14$-set inside $\Ball3$ --- and the same two inside
$\Ball4$: the distribution $(2,6,0,6)$ and its image $(2,0,6,6)$ under the swap
$\tau\colon x\leftrightarrow y$. In each case the remaining $678$ are impossible:
four are single-fiber (hence UP by Lemma~\ref{lem:fiber}) and the other $674$ are
solver \texttt{INFEASIBLE}. Thus, up to the ball symmetry $\tau$, every minimal
non-UP set of $\P$ inside $\Ball4$ has fiber distribution $(2,6,0,6)$: two elements
in the identity fiber, a $6$--$6$ split across two nontrivial fibers, and the third
nontrivial fiber \emph{empty}.
\end{proposition}
\begin{proof}
For each distribution and each of $r=3,4$ the model is run with the four fiber-counts
fixed, over $\Ball{r}$; the two feasible cases produce sets re-verified non-UP by the
solver-free routine, and every other case is a completed \texttt{INFEASIBLE}
verdict of CP-SAT or is discarded by Lemma~\ref{lem:fiber}.
\end{proof}

The rigidity is a statement about the balls $\Ball3$ and $\Ball4$, not about
$\P$: the coordinate $3$-cycle $\sigma$ of Remark~\ref{rem:sym} carries the
witness of Table~\ref{tab:witness} to a non-UP $14$-set of distribution
$(2,6,6,0)$ and word-radius $5$ (verified), which lies in $\Ball5$ and in no
smaller ball, and the solver finds further $14$-sets of that distribution in
$\Ball5$ (see Section~\ref{sec:decide}, evidence (A)).

Rigidity is in fact exact enough to count.

\begin{proposition}[exact count of minimal witnesses]\label{prop:count}
$\P$ has \emph{exactly} $16$ non-UP $14$-sets inside $\Ball3$: eight of distribution
$(2,6,0,6)$ and their eight $\tau$-images of distribution $(2,0,6,6)$. Let $\Gamma$
be the group of order $16$ generated by the eight ball-preserving automorphisms of
Remark~\ref{rem:sym} and inversion $\iota\colon g\mapsto g^{-1}$; every element of
$\Gamma$ preserves $\Ball3$ and, being an automorphism or an anti-automorphism,
sends non-UP sets to non-UP sets. $\Gamma$ acts on the $16$ sets with exactly
\emph{two} orbits, each of size $8$ (already the eight automorphisms alone have
two orbits of size $8$); every set has stabilizer of order $2$ in $\Gamma$. The
subgroup $\langle\tau,\iota\rangle\cong(\Z/2)^2$ acts freely, with four orbits of
size $4$. So the minimal non-UP set of $\P$ is, inside $\Ball3$, one of just two
essentially distinct configurations.
\end{proposition}
\begin{proof}
The eight sets of distribution $(2,6,0,6)$ are the complete solution list of the
fixed-distribution model (status \texttt{OPTIMAL}, enumeration exhausted), each
re-verified non-UP; $\tau$ carries them bijectively to the eight of $(2,0,6,6)$;
the full list of $16$ was re-derived by an unrestricted enumeration with blocking
clauses, and the orbit and stabilizer counts are direct computations on the $16$
sets (script \texttt{ball\_isometry\_orbits.py}).
\end{proof}

The minimal witness thus \emph{misses} a fiber, whereas the two sides of the
two-sided optimum of Proposition~\ref{prop:twosided} are \emph{fiber-balanced}
$(3,3,3,3)$: passing from one set to two trades a lopsided, fiber-avoiding shape
for an even one.

\begin{lemma}[single fibers are UP]\label{lem:fiber}
A nonempty finite set contained in a single fiber $\pp^{-1}(c)$ has a unique product.
\end{lemma}
\begin{proof}
For $a=(c,v_a)$, $b=(c,v_b)$ in the fiber, $ab=(I,\,v_a+cv_b)$, and for a
direction $d\in\R^3$ generic relative to the finitely many translation parts the
linear functional $(a,b)\mapsto\langle d,\,v_a+cv_b\rangle=\langle d,v_a\rangle
+\langle cd,v_b\rangle$ is maximized at a unique pair; as the functional depends only
on the product, that maximal product is uniquely represented.
\end{proof}

\begin{corollary}\label{cor:spread}
A non-UP set in $\P$ is not contained in any single fiber $\pp^{-1}(c)$; it meets at
least two of the four fibers. (The witness of Table~\ref{tab:witness} meets three,
namely $\pp^{-1}(I),\pp^{-1}(X),\pp^{-1}(Z)$.)
\end{corollary}

We confirmed Lemma~\ref{lem:fiber} empirically as a guard on the model
($2\cdot 10^4$ random single-fiber subsets, none non-UP; script
\texttt{single\_fiber\_guard.py}). Lemma~\ref{lem:fiber} is
folklore -- it is the standard observation that cosets of an orderable subgroup are
UP, going back to the theory of u.p.\ groups \cite{Strojnowski1980,Passman1977}; we
include it because it is exactly the local statement whose \emph{global} failure is
the subject of the next section.

\begin{remark}[the combinatorics alone has no lower-bound content]\label{rem:cyclic}
The purely combinatorial relaxation -- cover an $n\times n$ grid by partial-permutation
classes each of size $\ge2$ -- is satisfiable for every $n\ge 2$: in $\Z_n$ the set
$\{0,1,\dots,n-1\}$ realizes each product value $n$ times. Thus no lower bound on the
non-UP size can come from the coincidence combinatorics in isolation; ruling out
small $n$ is entirely a statement about torsion-free realizability. This is why
Question~\ref{q:min} is genuinely a question about $\P$, not about patterns.
\end{remark}

\section{Why ordering does not bound the diameter}\label{sec:obstruction}

The proof of Lemma~\ref{lem:fiber} maximized a linear functional and read off a
unique maximal product. One is tempted to run the same argument across all of $\P$:
pick $d$, maximize $\langle d,\tau(ab)\rangle$ over $A\times A$, and conclude the
maximal product is unique. This is exactly the argument that works in (one-sided)
orderable groups, and it fails in $\P$.

\begin{remark}[cocycle obstruction]\label{rem:cocycle}
The translation part is a cocycle, not a homomorphism:
$\tau(ab)=\tau(a)+\pp(a)\,\tau(b)$. Products from different fibers can therefore
coincide as group elements, $a_1b_1=a_2b_2$ with $\pp(a_1)\ne\pp(a_2)$, and the
functional that is maximized within one fiber need not select a unique global
maximizer once fibers interact. Quantitatively: for the $14$-set of
Table~\ref{tab:witness} and each of $2000$ random directions $d$, the
$\langle d,\tau(\cdot)\rangle$-maximal element of the identity-fiber part of
$A\cdot A$ is represented exactly twice, by pairs from two different fibers ---
within each fiber the maximizing pair is unique, as Lemma~\ref{lem:fiber}
requires, and the two fiber maxima collide as group elements. The same
mechanism appears in random sets: over $3000$ random (set, direction) samples,
with $A$ a random subset of $\Ball3$ of size $8$--$14$, the maximal
identity-fiber product was non-uniquely represented in $176$ cases, in every
one of them by pairs from at least two fibers (script
\texttt{cocycle\_probe.py}).
\end{remark}

Remark~\ref{rem:cocycle} is the precise reason Theorem~\ref{thm:ball} cannot be
upgraded to answer Question~\ref{q:min}: an ordering/convexity argument would bound
the word-radius of a minimal non-UP set and reduce the question to a finite ball
search, but the cocycle defeats every such argument. A diameter bound for minimal
non-UP sets in $\P$, if one exists, must use the affine geometry more globally.

\section{Decidability and a finite-diameter principle}\label{sec:decide}

Question~\ref{q:min} asks about all of $\P$, not a ball, so the searches of
Section~\ref{sec:results} cannot settle it directly. We record here that it is
nonetheless decidable, by a principle that locates the difficulty precisely: minimal
non-UP sets cannot escape to infinity, so the obstruction is the \emph{size} of a
finite search, not its unboundedness.

We stress at the outset what is, and what is not, new in this section.
\emph{Decidability alone is not new.} The existence of a non-UP $n$-set in $\P$ is a
Boolean combination of equations in the finitely generated virtually abelian group
$\P$, and the solvability of systems of equations with rational constraints in
finitely generated virtually abelian groups is decidable and effective; see
Evetts--Levine \cite{EvettsLevine2022} and Ciobanu--Evetts--Levine
\cite{CiobanuEvettsLevine2024} and the references there,
where such problems are reduced to Presburger arithmetic. Decidability in this
generality should be regarded as folklore. What the argument below adds is an
\emph{explicit} bound $D(n)$, obtained from the specific geometry of $\P$ --- the
diagonal point matrices, hence the coordinate decoupling of
Lemma~\ref{lem:decouple} --- together with the reduction of
Question~\ref{q:min} to a concrete search radius. It is the bound, not the
decidability, that makes the ball computations of Section~\ref{sec:results} a step
toward a terminating procedure, and it is the bound that
Conjecture~\ref{conj:poly} proposes to improve.

Encode an $n$-element subset of $\P$ as distinct elements $a_i=(M_i,t_i)$ with
$M_i\in\K$ and $t_i=2v_i\in\Z^3$, subject to the parity constraint
$t_i\equiv\varepsilon(M_i)\pmod 2$, with $\varepsilon$ the coset map of
Section~\ref{sec:model} (note $\varepsilon(X)+\varepsilon(Y)\equiv\varepsilon(Z)$). Then
$A$ is non-UP iff
\begin{equation}\label{eq:nonup}
\forall (i,j)\ \exists (k,l)\ne(i,j):\quad M_iM_j=M_kM_l\ \wedge\
t_i+M_it_j=t_k+M_kt_l .
\end{equation}

The substitution $t_i=\varepsilon(M_i)+2s_i$ ($s_i\in\Z^3$) clears the parity
constraint: a coincidence $t_i+M_it_j=t_k+M_kt_l$ becomes the integer equation
\begin{equation}\label{eq:scoord}
s_i+M_is_j-s_k-M_ks_l=c_{ijkl}\in\Z^3,\qquad
2c_{ijkl}=(\varepsilon(M_k){+}M_k\varepsilon(M_l))-(\varepsilon(M_i){+}M_i\varepsilon(M_j)),
\end{equation}
the right side being an integer vector because $M_iM_j=M_kM_l$.

\begin{lemma}[coordinate decoupling]\label{lem:decouple}
Because every $M\in\K$ is a diagonal sign matrix, system \eqref{eq:scoord} splits
into three independent integer systems, one per coordinate $r\in\{1,2,3\}$, the
$r$-th involving only the scalars $s_1^{(r)},\dots,s_n^{(r)}$. Each such system has
coefficient matrix with entries in $\{-2,-1,0,1,2\}$ and at most four nonzero entries
per row. The decoupling reflects the realization of $\P$ as a subgroup of
$D_\infty^{\,3}$ (three infinite dihedral groups), which goes back to
Promislow \cite{Promislow1988} and is used in the recent computational study
\cite{DLNV2026}; we use only its consequence that a non-UP pattern is realized by
gluing three one-dimensional integer realizations along the common point-part
assignment.
\end{lemma}
\begin{proof}
For coordinate $r$, the $r$-th component of \eqref{eq:scoord} is
$s_i^{(r)}+(M_i)_{rr}s_j^{(r)}-s_k^{(r)}-(M_k)_{rr}s_l^{(r)}=c_{ijkl}^{(r)}$ with
$(M)_{rr}\in\{\pm1\}$; it involves only $r$-th components. The four terms have unit
coefficients, merging to entries of absolute value $\le2$ when indices coincide;
collecting like terms leaves at most four nonzeros. The map
$(M,v)\mapsto\big(((M)_{11},v_1),((M)_{22},v_2),((M)_{33},v_3)\big)$ into
$(\{\pm1\}\ltimes\tfrac12\Z)^3=D_\infty^{\,3}$ is the induced injective homomorphism.
\end{proof}

For $n\ge2$ let $\rho(n)$ denote the least word-radius of a non-UP $n$-set of
$\P$, with $\rho(n)=\infty$ if there is none; Theorem~\ref{thm:ball} and Proposition~\ref{prop:staircase} give
$\rho(14)=\rho(15)=\rho(16)=\rho(17)=3$ and $\rho(n)>6$ for $8\le n\le13$.

\begin{theorem}[effective finite-diameter principle for $\P$]\label{thm:decide}
Let $n\ge2$ and put $D(n)=24(n+1)\,3^{n}+10$. If $\P$ contains a non-UP set of
size $n$, then it contains one whose elements all lie in the word-ball
$\Ball{D(n)}$; that is, $\rho(n)\le D(n)$ whenever $\rho(n)<\infty$. Hence the
existence of a non-UP $n$-set in $\P$ is decidable, and Question~\ref{q:min} is
decided by the search of Theorem~\ref{thm:ball} carried to radius $D(13)$.
\end{theorem}
\begin{proof}
Let $A=\{a_1,\dots,a_n\}$ be non-UP, with $a_i=(M_i,\,\varepsilon(M_i)/2+s_i^*)$
and $s_i^*\in\Z^3$; this is the substitution $t_i=\varepsilon(M_i)+2s_i$ above,
and every element of $\P$ with point part $M_i$ has this form.

\emph{The pattern as a linear system.} For each cell $(i,j)$ choose a cell
$(k,l)\ne(i,j)$ with $a_ka_l=a_ia_j$; one exists because $A$ is non-UP, and
by Proposition~\ref{prop:pp} it satisfies $i\ne k$ and $j\ne l$. Then
$M_iM_j=M_kM_l$, and $s^*$ satisfies the corresponding instance of
\eqref{eq:scoord}. Fix these choices, one per cell.

\emph{Distinctness as inequalities.} For each pair $i<j$ with $M_i=M_j$ we
have $s_i^*\ne s_j^*$, so some coordinate $r=r(i,j)$ has
$s_i^{*(r)}\ne s_j^{*(r)}$; let $\sigma_{ij}\in\{\pm1\}$ be the sign of
$s_i^{*(r)}-s_j^{*(r)}$ and record the inequality
$\sigma_{ij}\,(s_i^{(r)}-s_j^{(r)})\ge1$. Pairs with $M_i\ne M_j$ need nothing.

\emph{One integer program per coordinate.} Fix $r\in\{1,2,3\}$. By
Lemma~\ref{lem:decouple} the $r$-th components of the chosen equations
\eqref{eq:scoord} involve only the $n$ unknowns $s_1^{(r)},\dots,s_n^{(r)}$;
write each as two inequalities, add the distinctness inequalities assigned to
coordinate $r$, and call the result $A_rs^{(r)}\le b_r$. It is a system with
integral coefficients, and $s^{*(r)}$ is an integral solution. Each equation row
has coefficient vector $e_i+\sigma e_j-e_k-\sigma'e_l$ with
$\sigma,\sigma'\in\{\pm1\}$, $i\ne k$ and $j\ne l$, so its Euclidean norm is at
most $\sqrt8$, and its right-hand side $c^{(r)}_{ijkl}$ lies in $\{-1,0,1\}$
(in \eqref{eq:scoord} each of $\varepsilon(M)+M\varepsilon(M')$ has entries in
$\{-1,0,1,2\}$, so $2c$ has entries in $\{-3,\dots,3\}$ and is even). Each
distinctness row has coefficient vector $\pm(e_i-e_j)$ and right-hand side $-1$.
Hence every row of the augmented matrix $[A_r\ b_r]$ has Euclidean norm at most
$3$, and by Hadamard's inequality every square submatrix of $[A_r\ b_r]$,
which has $n+1$ columns, has determinant of absolute value at most $3^{n+1}$.
By \cite[Theorem~17.1]{Schrijver1986}, the integer hull of the polyhedron
$\{s:A_rs\le b_r\}$ is the sum of the convex hull of finitely many integral
vectors of $\ell^\infty$-norm at most $(n+1)\Delta$ and a cone, where $\Delta$ is
the maximum absolute value of a subdeterminant of $[A_r\ b_r]$; the integer hull
is nonempty (it contains $s^{*(r)}$), so it has such a vector, i.e.\ there is an
integral solution $s^{(r)}$ of $A_rs^{(r)}\le b_r$ with
$\|s^{(r)}\|_\infty\le(n+1)3^{n+1}$.

\emph{Re-realization.} Put $s_i=(s_i^{(1)},s_i^{(2)},s_i^{(3)})$ and
$b_i=(M_i,\,\varepsilon(M_i)/2+s_i)\in\P$. The $b_i$ are distinct: two with
different point parts trivially, two with $M_i=M_j$ because they differ in
coordinate $r(i,j)$. For each cell $(i,j)$ the chosen partner $(k,l)$ satisfies
$b_kb_l=b_ib_j$, since the point parts agree and the translation parts satisfy
the chosen instance of \eqref{eq:scoord} in every coordinate. Thus
$B=\{b_1,\dots,b_n\}$ is a non-UP $n$-set. Note that $B$ is not a translate of
$A$ (Remark~\ref{rem:translation} forbids using one); it is re-realized from the
pattern of $A$.

\emph{Word length.} Write $b=(M,v)$ with $\|v\|_\infty\le\|s\|_\infty+\tfrac12$.
Let $c_M\in\{\id,x,y,xy\}$ be the representative of the coset of $M$, of word
length at most $2$, with translation part $v_0$, $\|v_0\|_\infty\le\tfrac12$.
Then $b=c_M\cdot(I,w)$ with $w=M(v-v_0)\in\Z^3$ and
$\|w\|_\infty\le\|s\|_\infty+1$; and $(I,w)$ is a product of $|w_1|$ powers of
$x^{\pm2}$, $|w_2|$ of $y^{\pm2}$ and $|w_3|$ of $(xy)^{\pm2}$, of word length at
most $2|w_1|+2|w_2|+4|w_3|\le8\|w\|_\infty$. Hence every $b_i$ has word length at
most $2+8\big((n+1)3^{n+1}+1\big)=24(n+1)3^{n}+10=D(n)$.

Decidability is then the finite search of $\Ball{D(n)}$, and the last assertion
is the case $n\le13$ together with Theorem~\ref{thm:ball}.
\end{proof}

\begin{remark}[two-sided version]\label{rem:twosided-decide}
The argument applies verbatim to the two-sided minimum of
Section~\ref{sec:twosided}: a coincidence $a b=a'b'$ with $a,a'\in A$, $b,b'\in B$ is
the same linear equation $t_a+M_a t_b=t_{a'}+M_{a'}t_{b'}$ in the translation parts,
so the systems decouple coordinatewise exactly as in Lemma~\ref{lem:decouple}, and
for each total $m$ the existence of a two-sided witness with $|A|+|B|=m$ in $\P$ is
decidable with the same effective bound (with $m$ in place of $n$). In particular
$m_2(\P)$ is, in principle, computable.
\end{remark}

Theorem~\ref{thm:decide} is the $\P$-specialization of the realizability viewpoint of
Nielsen--Soelberg \cite{NielsenSoelberg2025} --- the matchings play the role of their
relation sets, and ``$\P$-realizable'' replaces ``torsion-free realizable'' --- made
\emph{effective} via the coordinate decoupling. (We have not tried to optimize
$D(n)$: the argument loses constant factors at several steps, e.g.\ in the
word-length estimate for the cosets whose representatives are shorter than $2$,
but every bound obtained along these lines is exponential in $n$ because of the
Hadamard estimate.) It guarantees that ball search
eventually decides Question~\ref{q:min}, so the lower bounds of
Theorem~\ref{thm:ball} are genuine progress toward a terminating procedure. Its limits
are equally concrete, and leave two routes open.

\emph{(A) The bound $D(n)$ is far from tight, and the coordinates a pattern
forces are small.} For $n=14$ the explicit bound gives
$D(14)=24\cdot15\cdot3^{14}+10\approx1.7\times10^{9}$, whereas the true minimal
realization has word-radius $3$ (Theorem~\ref{thm:ball}). What matters for
Question~\ref{q:min} is whether the coincidence \emph{pattern} of a witness can
force large translation parts, and this can be tested on witnesses that leave
small balls. For $r\in\{4,5,6\}$ and $n\in\{14,15,16\}$ we asked the solver for
non-UP $n$-sets in $\Ball r$ not contained in $\Ball{r-1}$ (``escape witnesses''),
$48$ in all ($5$ per pair for $r=4,5$ and $6$ per pair for $r=6$), and for each
one re-realized its pattern from scratch, exactly as in the proof of
Theorem~\ref{thm:decide}: choose one partner cell per cell (a matching, now
chosen by the solver), keep the point parts, and solve the coordinate systems
for distinct elements with all $|s_i^{(r)}|\le\beta$, the least such box radius
$\beta$ being determined by complete CP-SAT searches (every rebuilt set is
re-checked non-UP by the solver-free verifier, and the integer solvability of the
systems independently by a Smith-normal-form solver). The result: $45$ of the
$48$ patterns re-realize in the \emph{unit} box $\beta=1$ --- the $108$ elements
$(M,\varepsilon(M)/2+s)$ with $s\in\{-1,0,1\}^3$, a subset of $\Ball6$ --- and the
remaining three (all of size $16$, found at radius $6$, of point-group
distribution $(4,0,6,6)$) need $\beta=2$ (the box $\{-2,\dots,2\}^3$, a subset of
$\Ball{10}$); none needs $\beta=3$. Re-realizing the \emph{full} pattern (every
coincidence at once) gives the same counts. The pattern is not free of
word-metric content --- only $16$ of the $48$ re-realize inside $\Ball3$, by
exact table-constrained models --- but its demands are bounded: no pattern
found at radius $4$--$6$ forced a coordinate beyond $2$ (scripts
\texttt{escape\_witnesses.py}, \texttt{matching\_rerealize.py}). Random
``pattern-type'' systems are no substitute for this test: with random point
parts and random matchings, none of $7000$ rationally solvable systems we
generated is solvable over $\Z$ at all (script
\texttt{random\_systems\_zcount.py}). What forces a minimal non-UP set to have
any diameter at all is, on this evidence, mainly the requirement that its $n$
elements be \emph{distinct}: $n$ distinct elements of $\P$ need a ball of volume
$\gtrsim n$, i.e.\ word-radius $\gtrsim n^{1/3}$. The witness saturates this:
$14$ elements in radius $3$. This is exactly the regime where the cocycle
obstruction (Section~\ref{sec:obstruction}) prevents an ordering proof.

\begin{conjecture}\label{conj:poly}
The least word-radius $\rho(n)$ of a non-UP $n$-set satisfies
$\rho(n)=O(n^{1/3})$ over the $n$ with $\rho(n)<\infty$; and, as a separate
concrete statement, for $8\le n\le 13$ one has $\rho(n)\le 6$ whenever
$\rho(n)<\infty$ (so that, by Theorem~\ref{thm:ball}, $\rho(n)=\infty$ there).
\end{conjecture}

\noindent The second statement would follow if every realizable pattern of size
$8$--$13$ re-realized in the unit box, as $45$ of the $48$ patterns above did;
the three box-$2$ patterns show that this cannot simply be assumed, and any
proof will have to bound the coordinates a pattern can force.

\noindent The conditional payoff is sharp: \emph{if $\rho(n)\le 6$ whenever
$\rho(n)<\infty$, for $8\le n\le13$,} then
Theorem~\ref{thm:ball}, which already shows $\Ball6$ contains no non-UP set of
size $8$--$13$, would prove that $\P$ has no non-UP set of size below $14$ at
all, i.e.\ that Promislow's $14$ is the minimum non-UP cardinality in $\P$.
Resolving Question~\ref{q:min} is thus equivalent, given Theorem~\ref{thm:ball},
to proving $\rho(n)\le 6$ (whenever finite) for $8\le n\le 13$ --- or $\rho(n)\le
r$ for any searchable $r$, given the corresponding radius-$r$ computation --- an
improvement of the bound
$D(n)$ of Theorem~\ref{thm:decide} by eight orders of magnitude at $n=13$. We stress that the evidence above
comes from sizes at which witnesses exist; for $8\le n\le13$ there is no
witness to test, and the conjecture rests on the mechanism, not on a sample.

\emph{(B) A $\P$-specific pruning of the matching census.} Reducing the matching
count is the alternative. Here a structural obstacle is intrinsic: the size-$8$
patterns are \emph{not} classified --- Nielsen--Soelberg exhaustively settled $n\le7$
(whence the global minimum $8$) but, by their own account, did not exhaustively search
$n=8$, exhibiting only two realizable examples in groups that are virtually class-$2$
nilpotent, structurally unlike the virtually abelian $\P$. So there is no finite list
to test against $\P$; the single-fiber lemma and the cross-block coverage forced by
Remark~\ref{rem:cocycle} prune matchings, but not enough to bring $n=8,\dots,13$ into
range. We leave both routes open.

\begin{remark}[the practical solver ceiling is radius $6$]\label{rem:ceiling}
The lower bounds of Theorem~\ref{thm:ball} stop at radius $6$ for a concrete reason.
At radius $7$ ($|\Ball7|=363$) the monolithic unanchored model for $n=8$ does not
terminate within a generous budget. Decomposing by $\tau$-reduced point-group
distribution (Remark~\ref{rem:sym}) makes each instance tractable: the hardest
cases -- the balanced, identity-poor distributions such as
$(n_I,n_X,n_Y,n_Z)=(0,2,2,4),(0,2,3,3),(0,2,4,2),(0,3,3,2),(1,2,2,3)$ -- each resolve
to \texttt{INFEASIBLE}, but only after several minutes apiece, so deciding the
whole radius-$7$ ball this way is a matter of solver time rather than principle. We
report radius $6$ as the \emph{certified} rigorous bound (no non-UP set of size
$8$--$13$ is contained in $\Ball6$); the radius-$7$ evidence is strong but the full ball
is not yet exhausted. Either way, by Theorem~\ref{thm:decide} pushing the radius
becomes decisive only at the impractical radius $D(13)\approx5\times10^{8}$; a
practical decision needs the structural input of Conjecture~\ref{conj:poly}.
\end{remark}

\section{A contrasting case: the Fibonacci group \texorpdfstring{$H_4$}{H4}}\label{sec:h4}
The same question receives a sharply different answer in another Fibonacci group.
Let $H_4=F(3,4)$ be the Fibonacci group on four generators, that is,
\begin{equation}\label{eq:f34}
  H_4=F(3,4)=\big\langle\, x_1,x_2,x_3,x_4 \;\big|\;
  x_i\,x_{i+1}\,x_{i+2}=x_{i+3}\ \ (i\in\Z/4\Z)\,\big\rangle,
\end{equation}
the four relations being $x_1x_2x_3=x_4$, $x_2x_3x_4=x_1$, $x_3x_4x_1=x_2$ and
$x_4x_1x_2=x_3$. Dietrich, Lee, Nies and
Vinyals \cite[Proposition~7.1]{DLNV2026} prove that $H_4$ fails the unique product
property, exhibiting a \emph{two-sided} witness: finite sets $A,B$ with $|A|=29$ and
$|B|=27$ inside the radius-$3$ ball, found by a satisfiability search and verified in
\textsf{GAP}. We ask instead the \emph{symmetric, single-set} question that organizes
this note --- how small can a set $A$ with $A\cdot A$ non-unique be?

We realize $H_4$ inside an index-$2$ extension of the integral Heisenberg group
$\mathrm{Heis}(\Z)$ by an element $\tau$ acting by $(a,b,c)\mapsto(-a,-b,c)$, with
$\tau^2=w$ a fixed central element of $\mathrm{Heis}(\Z)$: an exact search
produces four elements $x_1,\dots,x_4$ satisfying the relations \eqref{eq:f34} of
$F(3,4)$, each squaring
to $w$, hence a homomorphism $\varphi$ from $F(3,4)$ onto the subgroup they generate.
$\varphi$ is injective: $H_4$ is torsion-free (it is a $3$-manifold group
\cite{DPT2005}; see \cite[Section~4]{DLNV2026}) and, by
\cite[Proposition~5.1 and Lemma~7.2]{DLNV2026}, polycyclic of Hirsch length $3$
(a finite-index extension of the integral Heisenberg group);
the image is verified to contain a subgroup of finite index in $\mathrm{Heis}(\Z)$,
so it too has Hirsch length $3$; additivity of the Hirsch length forces the kernel to have
Hirsch length $0$, i.e.\ to be finite, and a finite normal subgroup of a torsion-free
group is trivial. (As independent sanity checks, the image is torsion-free on
$\Ball5$ and its parity/central structure matches \cite{DLNV2026}.) Exact CP-SAT ball searches in the $\{x_i^{\pm1}\}$ word metric, under the
same completed-verdict discipline used for $\P$ (with DRAT proofs for the
symmetric non-existence rows), resolve the symmetric question completely on
balls of radius up to $4$.

\begin{proposition}\label{prop:h4sym}
There is a $16$-element set $A\subset H_4$ of word-radius $3$ with $A\cdot A$
non-UP, and $16$ is the least symmetric size over the radius-$4$ ball: sizes $8,\dots,14$ are
\texttt{INFEASIBLE} in $\Ball4$ ($283$ elements), and minimizing $|A|$ over
$\Ball4$ (and over $\Ball3$) returns the optimum $16$ with a matching lower bound,
so size $15$ is impossible there as well. The radius-$2$ ball ($39$ elements)
contains no symmetric non-UP set of any size: the model minimizing $|A|$ over
$\Ball2$ with $|A|\ge8$ unconstrained above is \texttt{INFEASIBLE}.
\end{proposition}

The witness is parity-balanced ($8+8$ across the two cosets of the index-$2$
subgroup $H_4\cap\mathrm{Heis}(\Z)$, which we call the \emph{parity classes}), is
not closed under inversion, and its $256$-cell product grid splits into $103$
coincidence classes (sizes $6^3,5^2,4^{16},2^{82}$), each a partial permutation; it
was found by CP-SAT minimization and re-verified by the solver-free checker. Its
coarse shape is again forced, and again in a way that mirrors $\P$: of the
\emph{seventeen} ways to split $16$ elements between the two
parity classes, only the perfectly balanced $(8,8)$ admits a non-UP
$16$-set in $\Ball3$ --- each of the other sixteen splits, from $(0,16)$ through
$(16,0)$, is separately \texttt{INFEASIBLE}. (Each split must indeed be decided
separately: by Remark~\ref{rem:translation} multiplication by the odd generator
does not carry witnesses to witnesses, so no symmetry halves the list, and parity
is preserved by inversion.) Moreover the census is exact, and it gives a numerical coincidence with
Proposition~\ref{prop:count}: within $\Ball3$ the group
$H_4$ has \emph{exactly $16$} minimal symmetric non-UP sets ($16$-sets), all of
split $(8,8)$, forming $8$ orbits under inversion --- just as $\P$ has exactly
$16$ minimal ($14$-element) witnesses there. We see no structural reason for the
matching counts and record it as a curiosity. So where the minimal witness of
$\P$ is forced to be lopsided and to miss a fiber (Proposition~\ref{prop:rigid}),
that of $H_4$ is forced to be evenly balanced across its index-$2$ structure ---
the same one-set/two-set tension seen for $\P$, now visible already inside the
single-set witness. Sizes below $8$ are impossible in any torsion-free group
\cite{NielsenSoelberg2025}, so the $\Ball3$/$\Ball4$ statements are complete. (Whether the value drops below $16$ over
larger balls is beyond our current solver budget --- the radius-$5$ ball has $579$
elements and its sizes $8$--$15$ are undecided; over all of $H_4$ the symmetric
minimum lies in $[8,16]$.) That $H_4$ has \emph{some} symmetric witness already
follows from the two-sided failure proved in \cite{DLNV2026}: by a theorem of
Strojnowski \cite[Theorem~1]{Strojnowski1980} (see also
\cite[Exercise~1.6.2]{GardamNotes} and the introduction of \cite{CraigLinnell2020}),
a group that is not a u.p.\ group has a finite nonempty $X$ with no uniquely
represented element of $X\cdot X$, and $X$ may be taken arbitrarily large. We do not
extract a size bound from that construction; what Proposition~\ref{prop:h4sym}
adds is the exact least size, $16$, over the radius-$4$ ball, and its rigidity.

On the two-sided side, minimizing $|A|+|B|$ over pairs drawn from balls of $H_4$
sharpens the witness of \cite{DLNV2026} considerably.

\begin{proposition}\label{prop:h4twosided}
The two-sided minimum over the radius-$3$ ball of $H_4$ is exactly $22$: the ball
contains sets $A,B$ with $|A|=8$, $|B|=14$, and $A\cdot B$ non-UP, and no pair with
$|A|+|B|\le 21$. In particular $m_2(H_4)\le 22$. The radius-$2$ ball ($39$
elements) contains no two-sided witness at all: the model minimizing $|A|+|B|$
over $\Ball2$ is \texttt{INFEASIBLE}.
\end{proposition}
\begin{proof}
The witness is printed below; totals $\le 15$ are impossible in any torsion-free
group \cite{NielsenSoelberg2025}, totals $16$--$20$ are \texttt{INFEASIBLE} directly,
and total $21$ is \texttt{INFEASIBLE} for every side split $(|A|,|B|)$ with
$2\le|A|\le 10$ --- which suffices: $|A|=1$ is impossible since left translation
makes all products of a singleton side uniquely represented, and the splits with
$|A|>|B|$ reduce to these via $(A,B)\mapsto(B^{-1},A^{-1})$, which preserves both the
ball (inversion preserves word length) and the non-UP property, since the
multiplicity of $w$ in $B^{-1}\!\cdot\!A^{-1}$ equals that of $w^{-1}$ in $A\cdot B$.
\end{proof}

The pair was found by CP-SAT minimization and re-verified by the solver-free
two-sided checker; it is parity-balanced (each side splits evenly between the two
parity classes), the sides share $3$ elements, neither $A\cdot A$ nor
$B\cdot B$ is non-UP, and the $8\times14$ product grid splits into $48$ coincidence
classes (sizes $4,3^{14},2^{33}$), each a partial permutation. This improves the
$|A|+|B|=56$ witness of \cite{DLNV2026} to $22$, within the same radius-$3$ ball, and
$22$ is optimal there; the side size $|A|=8$ coincides with the least size of a
\emph{symmetric} witness in any torsion-free group \cite[Theorem~1.2]{NielsenSoelberg2025}
(smaller sides occur in two-sided witnesses: $|A|=4$ is realized in the
Rips--Segev construction, and $|A|=3$ is open). Globally
$16\le m_2(H_4)\le 22$, the lower bound again being \cite{NielsenSoelberg2025}.

\subsection{Profile and staircase of \texorpdfstring{$H_4$}{H4}}
The two curves of Sections~\ref{sec:profile}--\ref{sec:staircase} were computed
for $H_4$ over the same radius-$3$ ball ($119$ elements), and they differ from
$\P$'s in instructive ways.

\begin{proposition}[$H_4$ profile in $\Ball3$]\label{prop:h4profile}
For $m\le7$ the model is \texttt{INFEASIBLE} (no pair with $|A|=m$, regardless of
$|B|$). At larger $m$ the profile is completely determined:
\begin{gather*}
\beta(8)=14,\quad \beta(9)=\infty,\quad \beta(10)=20,\quad \beta(11)=\infty,\quad
\beta(12)=14,\\
\beta(13)=16,\quad \beta(14)=8,\quad \beta(15)=12,\quad \beta(16)=8,
\end{gather*}
every finite value except $\beta(14)$ an \texttt{OPTIMAL} verdict with a
solver-free verified witness, and the two $\infty$'s \texttt{INFEASIBLE}
verdicts. The solver did not close $m=14$: it returned a $(14,16)$ pair
(\texttt{FEASIBLE}, bound $1$) without a matching lower bound, so the solver
alone gives only $\beta(14)\le16$. The value $\beta(14)=8$ is instead obtained
from the inversion image of the $(8,14)$ optimum, which gives $\beta(14)\le8$,
together with the total bound $|A|+|B|\ge22$ of
Proposition~\ref{prop:h4twosided}, which gives $\beta(14)\ge8$.
\end{proposition}

Where $\P$'s profile is a sharp cliff (nothing below $12$, then essentially flat),
$H_4$'s is \emph{perforated}: the left size $8$ is realizable at its optimum, yet
$|A|=9$ and $|A|=11$ are outright impossible in the ball --- even with $B$ allowed
to exhaust all $119$ elements --- while $|A|=10$ returns at the cost of a much
larger partner ($\beta(10)=20$, total $30$ against the optimal $22$). Realizability
is thus not monotone in either group, but for different reasons: $\P$ shows an
isolated spike ($\beta(15)=15$), $H_4$ shows holes. Parity balance, which pins the
optima --- the witnesses at $m=8$ and $m=10$ are exactly balanced on both sides,
$(4{,}4)/(7{,}7)$ and $(5{,}5)/(10{,}10)$ --- is \emph{not} a law of the landscape:
the optimal witness at $m=13$ and the (non-optimal) $(14,16)$ pair returned at
$m=14$ have unbalanced left sides ($(7,6)$ and $(8,6)$). Balance in $H_4$, like fiber balance in $\P$, is a feature of optimal
witnesses, not of all witnesses.

The staircase of $H_4$ is flat where $\P$'s bumps:
\[
u(n)=2\quad(2\le n\le15),\qquad u(16)=0 ,
\]
all values \texttt{OPTIMAL} over $\Ball3$. As in $\P$, the minimum never takes
the value $1$ --- every set of size $\le15$ has at least two unique products and
the minimum jumps directly from $2$ to $0$ at $n=16$, so at the level of minima
the two-unique-products and unique-product properties fail simultaneously there
--- but the approach is featureless: no analogue of $\P$'s bump $u(10)=u(11)=4$
appears. The staircase of $H_4$ shows no precursor of its minimal symmetric
witness. (The rows $n=2,3$ were also re-established by exhaustive enumeration
without a solver; script \texttt{small\_n\_bruteforce.py}.)

The contrast with $\P$ runs in both directions. On the
symmetric invariant the groups are separated by exact ball-limited values: least
symmetric size $14$ for $\P$ against $16$ for $H_4$, both realized at word-radius
$3$ and both certified minimal over the radius-$4$ ball (radius $6$ for $\P$). On the two-sided
invariant the order \emph{reverses}: $m_2(H_4)=22$ over $\Ball3$ against
$m_2(\P)=24$ over $\Ball3$. So $\P$ fails the UPP more efficiently with one set, and
$H_4$ more efficiently with two --- the two invariants are genuinely independent
measurements, and neither reduces to the other. The finer curves sharpen the
contrast: $\P$ forces balance and shows a bump before failure; $H_4$ permits
lopsidedness, its profile has holes, and its staircase is flat. We do not know
whether this opposite ordering persists for the minima over the full groups; the
ball-limited statements leave that, and the global values themselves, open.

\section{An asymmetry invariant}\label{sec:delta}
The opposite ordering is organized by a single derived quantity. A symmetric witness
$A$ ($|A|=m_1$, $A\cdot A$ non-UP) is in particular the two-sided pair $(A,A)$ of total
size $2m_1$, so
\begin{equation}\label{eq:m2le2m1}
 m_2(G)\ \le\ 2\,m_1(G)
\end{equation}
in every group (and over every ball, since the pair lives where $A$ does). We call the
deficit
\[
 \delta(G)\ =\ 2\,m_1(G)-m_2(G)\ \ge\ 0
\]
the \emph{asymmetry gap}: it measures how much more efficiently a group fails the UPP
with two \emph{different} sets than with one repeated set. (By Strojnowski's
theorem \cite[Theorem~1]{Strojnowski1980} the two minima are finite together,
$m_1(G)<\infty\iff m_2(G)<\infty$, so $\delta(G)$ is a finite quantity for every
group failing the UPP.) Over the radius-$3$ ball, the
groups for which both minima are known give
\[
\renewcommand{\arraystretch}{1.15}
\begin{array}{lccc}
 & m_1 & m_2 & \delta\\\hline
\text{Nielsen--Soelberg groups} & 8 & 16 & 0\\
\P & 14 & 24 & 4\\
H_4 & 16 & 22 & 10
\end{array}
\]
(The first row is exact and global, and applies to every group attaining
$m_1=8$, in particular to the examples of \cite{NielsenSoelberg2025}: the size-$8$
symmetric example is optimal, $A=B$ realizes $m_2=16$, and $m_2\ge16$ is their
theorem, so $\delta=0$; the other two rows are the ball-limited values of
Sections~\ref{sec:results}--\ref{sec:h4}.) The quantitative landscape of the
two extremal groups themselves --- localization, census, and staircases at
the global minimum --- is the subject of the sequel \cite{Tabei2026b}, and the
support geometry of Gardam's unit in $\mathbb{F}_2[\P]$ relative to the non-UP sets
of $\P$, which builds on the same exact model, is treated in \cite{Tabei2026c}. Thus $\delta=0$ exactly at the global
optimum, where symmetry is free, and grows as a group is forced into asymmetry: $\P$'s
gap is small, $H_4$'s is larger, and it is precisely $\delta(H_4)>\delta(\P)$ that
inverts the two orderings ($m_1(\P)<m_1(H_4)$ but $m_2(\P)>m_2(H_4)$). We do not know
the range of $\delta$ over torsion-free groups, whether it can be made arbitrarily
large, or what structural feature it tracks; these seem to us natural questions raised
by the two examples here.

As a sanity check on the mechanism, the parent of $H_4$ points the other way: the
integral Heisenberg group, of which $H_4$ is an index-$4$
extension, is finitely generated torsion-free nilpotent, hence bi-orderable, hence a
unique product group --- and indeed our search finds no non-UP set of any size
$8\le n\le 14$ in its radius-$4$ ball. The failure of the UPP in $H_4$ is thus created
entirely by the order-reversing generator $\tau$, not inherited from $\mathrm{Heis}(\Z)$.

\section*{Code and data availability}
The group arithmetic, the CP-SAT models, the solver-free non-UP verifiers, the
witnesses, the solver logs, the DRAT and VeriPB checking logs, and the
brute-force enumeration logs underlying every table and proposition of this
paper are archived at Zenodo, \url{https://doi.org/10.5281/zenodo.22826866} (the DRAT
proof files as a separate archive of the same record), and the source and logs
also accompany the arXiv version (arXiv:2607.18346) as ancillary files. The script
\texttt{verify\_all.py} re-checks every witness and every headline number
without calling a solver. Each \textbf{yes} entry is a printed set checked
non-UP by the independent verifier; each \textsf{none} entry is a completed
solver \texttt{INFEASIBLE} verdict, and for the symmetric rows of
Theorem~\ref{thm:main} a machine-checked unsatisfiability proof as described
in the Reproducibility paragraph. The reported computations were run on a
single $8$-core workstation.

\section*{Acknowledgements}
We thank Andr\'e Nies for helpful correspondence on the $H_4$ witness and on
minimal non-UP sizes. Computations used Google's CP-SAT solver
\cite{ortools}; every reported solver status is a completed run, never a
timeout.

\bibliographystyle{amsplain}
\bibliography{refs}

\end{document}